\documentclass[12pt,a4paper]{amsart}
\usepackage[latin1]{inputenc}
\usepackage{amssymb, amsmath}
\usepackage{geometry}
\usepackage{enumerate}
\newtheorem{theorem}{Theorem}[section]
\newtheorem{definition}[theorem]{Definition}
\newtheorem{lemma}[theorem]{Lemma}
\newtheorem{corollary}[theorem]{Corollary}

\newtheorem{remark}[theorem]{Remark}

\begin{document}
\title[Optimal r.i.  Sobolev embeddings in mixed norm spaces]{Optimal rearrangement invariant Sobolev embeddings in mixed norm spaces}

\author{Nadia Clavero}
\address{Department of  Applied Mathematics and Analysis, University of Barcelona, Gran Via 585, E-08007 Barcelona, Spain}
\email{nadiaclavero@ub.edu}

\author{Javier Soria}
\address{Department of  Applied Mathematics and Analysis, University of Barcelona, Gran Via 585, E-08007 Barcelona, Spain}
\email{soria@ub.edu}

\thanks{}


\thanks{Both authors have been partially supported by the  grants  MTM2013-40985-P (Spanish Government) and 2014SGR289 (Catalan Autonomous Government).}

\begin{abstract}
We improve the  Sobolev-type embeddings
 due to Gagliardo \cite{Gagliardo} and  Nirenberg~\cite{Nirenberg}
in the setting of rearrangement invariant   (r.i.) spaces. In particular we concentrate on seeking   the optimal domains
and the optimal ranges for these embeddings between r.i.\ spaces and mixed norm spaces. As a consequence, we prove that  the classical estimate for the standard Sobolev space $W^{1}L^{p}$ by Poornima~\cite{Poornima}, O'Neil~\cite{Oneil-1963} and
 Peetre~\cite{Peetre} ($1 \leq  p < n$), and by
Hansson~\cite{Hansson},  Brezis and Wainger~\cite{Brezis-Wainger} and Maz'ya~\cite{Mazya} ($p=n$)
 can be further strengthened by considering mixed norms on the target spaces.
 
\end{abstract}
\maketitle

\thispagestyle{empty}

\section{Introduction}

Let $n\in\mathbb N,$  with $n\geq 2$, and let $I\subset \mathbb R$ be an interval. The Sobolev space  $W^{1}L^{p}(I^{n}),$ $1\leq p\leq \infty,$  consists of all functions in $L^{p}(I^{n})$ whose first-order distributional derivatives also belong to $L^{p}(I^{n})$.  The classical Sobolev embedding theorem claims that
\begin{align}\label{eq: inclusion sobolev}
W^{1}L^{p}(I^{n})\hookrightarrow L^{pn/(n-p)}(I^{n}),\ \ 1\leq p<n.
\end{align}
 Sobolev \cite{Sobolev} proved this embedding for $p>1,$ but his method, based on integral representations, did not work when $p=1$. That case was settled affirmatively by Gagliardo \cite{Gagliardo} and  Nirenberg~\cite{Nirenberg}, who first  observed that 
\begin{align}\label{eq: Gagliardo-Niremberg}
W^{1}L^{1}(I^{n})\hookrightarrow \mathcal{R}(L^{1},L^{\infty}),
\end{align}
(see Definition~\ref{definition: mixed norm spaces}) and then, using an iterated form of H\"older's inequality,  completed the proof; i.e.,
 \begin{align*}
 W^{1}L^{1}(I^{n})\hookrightarrow \mathcal{R}(L^{1},L^{\infty})\hookrightarrow L^{n^{\prime}}(I^{n}),
\end{align*}
where $n^{\prime}$ denotes the conjugate exponent of $n$, i.e., $1/n+1/n^{\prime}=1$. 

Later, a new approach based on properties of mixed norm spaces was introduced by Fournier~\cite{Fournier} and was subsequently developed, via different
methods, by various authors, including Blei and Fournier~\cite{Blei-Fournier}, Milman \cite{Milman-kfunctional-mixed-norm}, Algervik and Kolyada~\cite{Robert-Viktor} and Kolyada~ \cite{Viktor2012, Viktor2013}. To be more precise, the central part of Fournier's work was  to study embeddings between mixed norm spaces and  Lorentz spaces $L^{p,q}$ (see  Sections~\ref{sec: preliminaries}  for further details on Lorentz spaces). Specifically, he proved that 
\begin{align*}
\mathcal{R}(L^{1},L^{\infty})\hookrightarrow L^{n^{\prime},1}(I^{n}),
\end{align*}
and then taking into account \eqref{eq: Gagliardo-Niremberg}, he obtained the following improvement of \eqref{eq: inclusion sobolev}:
 \begin{align}\label{eq: intro W1L1->Xop}
W^{1}L^{1}(I^{n})\hookrightarrow L^{n^{\prime},1}(I^{n}).
\end{align}
The embedding  \eqref{eq: intro W1L1->Xop} is due to  Poornima~\cite{Poornima}, and it can be also traced in the works of O'Neil~\cite{Oneil-1963} and Peetre~\cite{Peetre}.

 A thorough study of mixed norm spaces has been recently considered in \cite{Clavero-Soria}.  In particular,  extending the mixed norm estimates due to Fournier \cite{Fournier} to more general r.i.\ spaces, we have obtained a description of the largest mixed norm space of the form $\mathcal{R}(X,L^{\infty})$ that is  continuously embedded into  a fixed r.i.\ space.

In recent years, extensions  of \eqref{eq: inclusion sobolev} for more general  rearrangement invariant (r.i.) spaces have been extensively studied   by various authors, including  Edmunds, Kerman and  Pick~\cite{Edmunds-Kerman-Pick}, Kerman and Pick~\cite{Kerman-Pick} and Cianchi~\cite{Cianchi}. To be more specific, Kerman and Pick~\cite{Kerman-Pick}  were interested  on seeking  necessary and sufficient conditions for the  following embeddings involving  r.i.\ spaces to hold:
\begin{align}\label{eq: intro W1Z->Xop}
W^{1}Z(I^{n})\hookrightarrow X^{\textnormal{op}}(I^{n}).
\end{align}
This characterization was then exploited to study the optimal domain-range problems for the embedding \eqref{eq: intro W1Z->Xop}, within the class of r.i.\ spaces.

All these works provide us a strong motivation to   consider  \eqref{eq: Gagliardo-Niremberg} for more general r.i.\ spaces,  as well as to describe the optimal domain and the optimal range for this embedding between r.i.\ spaces and  mixed norm spaces.
 
The paper is organized as follows. In Sections~\ref{sec: preliminaries} and \ref{mixed norm spaces}, we present some basic properties of r.i.\ spaces and mixed norm spaces we shall need for our work. 

 Section~\ref{sec: Sobolev embedding in mixed norm spaces} is devoted to  study the Sobolev embedding of the form
\begin{align}\label{eq: introduction W1Z->R(X,Loo)}
W^{1}Z(I^{n})\hookrightarrow \mathcal{R}(X,L^{\infty}),
\end{align}
extending the classical estimate \eqref{eq: Gagliardo-Niremberg}. Following the ideas of  \cite{Kerman-Pick},  we establish the equivalence between \eqref{eq: introduction W1Z->R(X,Loo)} and the  boundedness of a Hardy type operator (see Theorem~ \ref{theorem: reduccion medida finita}). This relation will be  a key tool to determine the optimal domain and the optimal range for \eqref{eq: introduction W1Z->R(X,Loo)} between r.i.\ spaces and  mixed norm spaces. 

 After this  discussion, our analysis focuses  on giving explicit constructions of such optimal spaces. In particular, Theorem~\ref{teo: W1Z->R(X,Loo) rango}  provides a characterization of the smallest space of the form $\mathcal{R}(X,L^{\infty})$ in \eqref{eq: introduction W1Z->R(X,Loo)}, once the r.i.\ space $Z(I^{n})$ is given.  Finally, for a fixed mixed norm space $\mathcal{R}(X,L^{\infty})$, Theorem~\ref{teo: W1Z->R(X,Loo) dominio} describes the largest  r.i.\ space $Z$ for which \eqref{eq: introduction W1Z->R(X,Loo)} holds.

 All these results are then employed to establish classical Sobolev  embeddings in the context of  mixed norm spaces. Thus, for instance, we recover \eqref{eq: Gagliardo-Niremberg} and, as a new contribution,  we show  that    $\mathcal{R}(L^{1},L^{\infty})$ is  the smallest mixed norm space of the form    $\mathcal{R}(X,L^{\infty})$ satisfying \eqref{eq: Gagliardo-Niremberg}.
  
As we have mentioned before,  the optimal range problem for the Sobolev embedding was studied in \cite{Kerman-Pick} within the class of r.i.\ spaces. In particular, for a fixed r.i.\ domain space $Z(I^{n})$ they determined  the smallest r.i.\ space, namely  $X^{\textnormal{op}}(I^{n}),$ satisfying \eqref{eq: intro W1Z->Xop}.  Motivated by this problem, in Section~\ref{sec: Comparison with the optimal r.i. space}  we compare the optimal r.i.\ range space  with the optimal mixed norm space, and we prove in  Theorem~\ref{theorem: R(X,Loo)->Lubos}    that the  following chain of embeddings holds:
 \begin{align*}
 W^{1}Z(I^{n})\hookrightarrow \mathcal{R}(X,L^{\infty})\hookrightarrow  X^{\textnormal{op}}(I^{n}),
 \end{align*}
with $\mathcal{R}(X,L^{\infty})$   the mixed norm space constructed in Theorem~\ref{teo: W1Z->R(X,Loo) rango}. Consequently, it turns out that it is still  possible to further  improve  the classical Sobolev embeddings by means of mixed norm spaces.  
 
Some remarks about the notation: The measure of the unit ball in $\mathbb R^{n}$ will be represented by $\omega_{n}$. As usual, we  use the symbol $A\lesssim B$ to indicate that there exists a universal positive constant $C$, independent of all
important parameters, such that $A\leq CB$. The equivalence $A\approx B$ means that $A\lesssim B$ and $B\lesssim A$. Finally, the arrow  $\hookrightarrow$ stands for a continuous embedding.

\section{Preliminaries}\label{sec: preliminaries}
We collect in this section some basic notations and results that will be useful in what follows.

Let $n\in\mathbb N,$ with $n\geq 1$ and let $I\subset \mathbb R$ be an interval having Lebesgue measure $|I|=1$. We write  $\mathcal{M}(I^{n})$ for the set of all real-valued measurable functions on $I^{n}$ and $\mathcal{M}_{+}(I^{n})$ for the non-negative ones. 

Given $f\in\mathcal{M}(I^{n}),$ its distribution function $\lambda_{f}$  is defined by
\begin{align*}
\lambda_{f} (t) = |\{x\in I^{n}: |f (x)| > t\}|,\ \ t\geq  0,
\end{align*}
and the decreasing rearrangement $f^{*}$  of $f$ is defined as
\begin{align*}
f^{*}(t) =  \inf\{s\geq 0: \lambda_{f}(s)\leq t\},\ \ t\geq 0.
\end{align*}
It is easily seen that  if $h$ is a nonnegative and decreasing functions on $(0,1)$ then
\begin{align}\label{eq: radial function}
g(x)=h(\omega_{n}|x|^{n}), \text{ a.e. } x \Longrightarrow g^{*}=h^{*}.
\end{align}

As usual, we shall use the notation $f^{**}(t)=t^{-1}\int^{t}_{0}f^{*}(s)ds$. A basic property of rearrangements is the Hardy-Littlewood inequality (cf. e.g. \cite[Theorem II.2.2]{Bennett}), which says: 
\begin{align*}
\int_{I^{n}}|f(x)g(x)|dx\leq \int^{1}_{0}f^{*}(t)g^{*}(t)dt,\ \ f,g\in\mathcal{M}(I^{n}).
\end{align*}

A Banach function norm $\rho$ is a mapping $\rho: \mathcal{M}_{+}(I^{n})\rightarrow [0, \infty]$ such that the following properties hold:
 \begin{enumerate}[ ({A}1)]
	\item $\rho(f)=0 \Leftrightarrow f=0\ \textnormal{a.e.}$, $\rho(f+g)\leq \rho(f)+\rho(g)$, $\rho(\alpha f)=\alpha \rho(f)$, for $\alpha\geq 0$; \vspace{0.2 cm}
	\item if $0\leq g\leq f\ \textnormal{a.e.,}$ then  $ \rho(g)\leq \rho(g)$; \vspace{0.2 cm}
	\item if $0\leq f_{j}\uparrow f\ \textnormal{a.e.,}$ then $\rho(f_{j})\uparrow \rho(f)$;
	\item $\rho(\chi_{I^{n}})<\infty$;  \vspace{0.2 cm}
	\item $\int_{I^{n}}|f(x)|dx\lesssim \rho(f)$.
\end{enumerate}
By means of $\rho$, a Banach function space  $X(I^{n})$ can be defined:
\begin{align*}
X(I^{n})=\Big\{f\in \mathcal{M}(I^{n}): \rho(|f|)<\infty \Big\}.
\end{align*}
For each $f\in X(I^{n})$, we define $\|f\|_{X(I^{n})}=\rho(|f|)$.

A Banach function norm is rearrangement invariant if $\|f\|_{X(I^{n})}=\|g\|_{X(I^{n})},$ for every pair of functions $f$, $g$ which are equimeasurable, that is, $\lambda_{f} = \lambda_{g}$. This means that the norm of a function $f$ in $X(I^{n})$ depends only on its distribution function. In this case, we say that the Banach function space $X(I^{n})$ is rearrangement invariant (briefly an r.i.\ space).

The Lebesgue spaces $L^{p}(I^{n})$, with $1\leq p\leq \infty$, endowed with the standard norm, are the simplest example of r.i.\  spaces.  We shall also work with the Lorentz spaces, defined either for  $p=q=1$ or $p=q=\infty$, or $1<p<\infty$ and $1\leq q\leq \infty$ as 
\begin{align*}
L^{p,q}(I^{n})=\Big\{f\in \mathcal{M}(I^{n}): \|f\|_{L^{p,q}(I^{n})}<\infty\Big\},
\end{align*}
where
\begin{align*}
\|f\|_{L^{p,q}(I^{n})}=\Big\|t^{1/p-1/q}f^{*}(t)\Big\|_{L^{q}(0,1)},
\end{align*}
and, more generally, with the Lorentz-Zygmund spaces, defined for $1\leq p\,,q\leq \infty$ and $\alpha\in \mathbb R$ as
\begin{align*}
L^{p,q;\alpha}(I^{n})=\Big\{f\in \mathcal{M}(I^{n}): \|f\|_{L^{p,q;\alpha}(I^{n})}<\infty\Big\},
\end{align*}
where
\begin{align*}
\|f\|_{L^{p,q;\alpha}(I^{n})}=\Big\|t^{1/p-1/q}[1+\log(1/t)]^{\alpha}f^{*}(t)\Big\|_{L^{q}(0,1)}.
\end{align*}
Observe that $L^{p,p}(I^{n})=L^{p}(I^{n})$ and $L^{p,q;0}(I^{n})=L^{p,q}(I^{n})$. Let us also mention, for the sake of completeness, that the quantities $\|\cdot\|_{L^{p,q}(I^{n})}$ and $\|\cdot\|_{L^{p,q;\alpha}(I^{n})}$ are in general only quasi-norms, since they may fail to satisfy the triangle inequality. In most cases, they can be turned into equivalent norms replacing $f^{*}$ by $f^{**}$ in corresponding definitions. However,   when the weights $t^{1/p-1/q}$ or $t^{1/p-1/q}[1+\log(1/t)]^{\alpha}$ are non-increasing (and hence, in all   cases we are going to consider), then $\|\cdot\|_{L^{p,q}(I^{n})}$ and $\|\cdot\|_{L^{p,q;\alpha}(I^{n})}$ are norms.

Given an r.i.\ space $X(I^{n}),$ the set 
\begin{align*}
X^{\prime}(I^{n})=\bigg\{f\in\mathcal{M}(I^{n}):\int_{I^{n}}|f(x)g(x)|dx<\infty,\ \textnormal{for any $g\in X(I^{n})$}\bigg\},
\end{align*}
equipped with the norm
\begin{align*}
\|f\|_{X^{\prime}(I^{n})}=\sup_{\|g\|_{X(I^{n})}\leq 1}\int_{I^{n}}|f(x)g(x)|dx,
\end{align*}
is called the associate space of $X(I^{n})$. It turns out that $X^{\prime}(I^{n})$ is again an r.i.\ space \cite[Theorem~I.2.2]{Bennett}. 
The fundamental function of an r.i.\ space $X(I^{n})$ is given by
\begin{align}\label{eq: fundamental function}
\varphi_{X}(t)=\|\chi_{E}\|_{X(I^{n})},
\end{align}
where $|E|=t$ and $\chi_{E}$ denotes the characteristic function of the set $E\subset I^{n}$. 

A basic tool for working with r.i.\ spaces is the Hardy-Littlewood-P\'olya Principle  \cite[Theorem~II.2.2]{Bennett}, which asserts that if $f\in X(I^{n})$ and 
\begin{align*}
\int^{t}_{0}g^{*}(s)ds\leq \int^{t}_{0}f^{*}(s)ds,\ \ 0<t<1,
\end{align*}
then $g\in X(I^{n})$ and $\|g\|_{X(I^{n})}\leq \|f\|_{X(I^{n})}$.

For later purposes, let us recall the Luxemburg  representation theorem  \cite[Theorem~II.4.10]{Bennett}. It says that given an r.i.\ space $X(I^n)$, there exists another r.i.\ space $\overline{X}(0,1)$  such that 
$$f\in X(I^{n})\Longleftrightarrow f^{*}\in\overline{X}(0,1),$$
 and in this case $\|f\|_{X(I^{n})}=\|f^{*}\|_{\overline{X}(0,1)}$.

Next, we  recall the definition of the  Boyd indices of an r.i.\ space. First we introduce
the dilation operator: if $f\in\mathcal{M}(0,1)$ and $t>0,$
\begin{align*}
E_{t}f(s)=\begin{cases}
f(s/t),& \textnormal{if $0\leq s\leq \min (1,t),$}\\
0,&\textnormal{otherwise.}
\end{cases}\
\end{align*}
It is well-known that the operator $E_{t}$ is bounded on $\overline{X}(0, 1),$ for every r.i.\ space $X(I^{n})$ and for every $t>0$ (see e.g. \cite[Proposition~III.5.11]{Bennett}). 

By means of the norm
of $E_{t}$ on $\overline{X}(0,1)$, denoted as $h_{X}(t),$ we define the lower and upper Boyd indices of $X(I^{n})$  as
\begin{align*}
\underline{\alpha}_{X}=\sup_{0<t<1}\frac{\log h_{X}(t)}{\log(t)}\ \ \textnormal{and}\ \  \overline{\alpha}_{X}=\inf_{1<t<\infty}\frac{\log h_{X}(t)}{\log(t)}.
\end{align*}
It is easy to see that $0\leq \underline{\alpha}_{X}\leq \overline{\alpha}_{X}\leq 1$. For instance, if $X=L^{p,q}$, then   $h_{X}(t) = t^{1/p}$ and thus $\underline{\alpha}_{X}=\overline{\alpha}_{X}=1/p$. Furthermore,  for later purposes, let us  emphasize that $\overline{\alpha}_{L^{\infty,p;-1}}=0,\ \ 1<p<\infty$ (for  more details see~\cite{Bennett, Bennett-Rudnick}).

Let us next recall some special results from Interpolation Theory, which we shall need in what follows (for further information on this topic see \cite{Bennett, Bergh}).

Given a pair of compatible Banach spaces $(X_{0}, X_{1})$ (compatible in the sense that they are continuously embedded into a common Hausdorff topological vector space), their $K$-functional is defined, for each $f\in X_{0}+X_{1},$ by
\begin{align*}
K(f,t; X_{0},X_{1}):=\inf_{f=f_{0}+f_{1}}(\|f_{0}\|_{X_{0}}+t\|f_{1}\|_{X_{1}}),\ \ t>0.
\end{align*}

The fundamental result  concerning the $K$-functional is the following \cite[Theorem~V.1.11]{Bennett}:
\begin{theorem}\label{theorem: interpolation operator} Let $(X_{0},X_{1})$ and $(Y_{0},Y_{1})$ be two compatible pairs of Banach spaces and let $T$ be a sublinear operator satisfying
\begin{align*}
T\:: X_{0}\rightarrow Y_{0},\ \  \textnormal{and}\ \  T\:: X_{1}\rightarrow Y_{1}.
\end{align*}
Then, there exists a constant $C>0$ (depending only on the norms of $T$ between $X_{0}$ and $Y_{0}$ and between $X_{1}$ and $Y_{1}$) such that
\begin{align*}
K(Tf,t; Y_{0},Y_{1})\leq C K(Cf,t; X_{0}, X_{1}),\ \ \text{for every $f\in X_{0}+X_{1}$ and $t>0$.}
\end{align*} 
\end{theorem}

The $K$-functional for pairs of Lorentz spaces $L^{p,q}(I^{n})$ is given, up to equivalence, by the following result.

\begin{theorem}\label{holmfor}
(Holmstedt's formulas \cite[Theorem 4.2]{Holmstedt})\label{theorem: Holmstedt formulas} Let  $p_{0}=q_{0}=1$ or $1<p_{0}<\infty$ and $1\leq q_{0}<\infty$. Let $1/\alpha=1/p_{0}-1/p_{1}$. Then, 
\begin{align*}
K(f,t; L^{p_{0},q_{0}}(I^{n}), L^{\infty}(I^{n}))\approx \biggl(\int^{t^{p_{0}}}_{0}\bigl[s^{1/p_{0}-1/q_{0}}f^{*}(s)\bigr]ds\biggr)^{1/q_{0}},\ \ \textrm{for $t>0$.}
\end{align*}
\end{theorem}

The first-order Sobolev space  built upon an r.i.\ space $Z(I^{n})$ is defined as 
\begin{align*}
W^{1}Z(I^{n})=\big\{u\in Z(I^{n}): |\nabla u|\in Z(I^{n})\big\},
\end{align*}
endowed with the norm 
\begin{align*}
\|u\|_{W^{1}Z(I^{n})}= \|u\|_{Z(I^{n})}+\||\nabla u|\|_{Z(I^{n})}.
\end{align*} 
Here, $\nabla u$ stands for the gradient of $u$ and $|\nabla u|=(\sum^{n}_{i=1}u^{2}_{x_{i}})^{1/2}$. Observe  that if $|D^{1}u|$  denotes the Euclidean length of $(u, \nabla u)$ as an element of $\mathbb R^{n+1},$ then
\begin{align*}
\|u\|_{W^{1}Z(I^{n})}\approx \||D^{1} u|\|_{Z(I^{n})}.
\end{align*}
Concerning the $K\textnormal{-functional}$ for a couple of Sobolev spaces, we mention the work of DeVore and Scherer \cite{DeVore-Scherer}, who proved that, for every $u\in W^{1}L^{1}(I^{n}),$
\begin{align}\label{eq: Kfunctional Sobolev}
K(u,t; W^{1}L^{1}(I^{n}), W^{1}L^{\infty}(I^{n}))\approx \int^{t}_{0}|D^{1}u|^{*}(s)ds,\ \ t>0.
\end{align}
For later purposes,  we would like to observe that, using \eqref{eq: Kfunctional Sobolev}, the reiteration theorem \cite[Theorem~V.2.4]{Bennett} and Theorem~\ref{holmfor}, if either $p_{0}=q_{0}=1,$ or $1<p_{0}<p_{1}<\infty$ and $1\leq q_{0},q_{1}<\infty,$ then, for any $t>0,$
\begin{align}\label{eq: Holmstedt sobolev}
K(u,t; W^{1}L^{p_{0},q_{0}}(I^{n}), W^{1}L^{p_{1},q_{1}}(I^{n}))&\approx \bigg(\int^{t^{\alpha}}_{0}\bigl[s^{1/p_{0}-1/q_{0}}|D^{1} u|^{*}(s)\bigr]^{q_{0}}ds\biggr)^{1/q_{0}}\\
\nonumber &\qquad+ t \biggl(\int^{1}_{t^{\alpha}}\bigl[s^{1/p_{1}-1/q_{1}}|D^{1} u|^{*}(s)\bigr]^{q_{1}}ds\biggr)^{1/q_{1}},
\end{align}
 where $\alpha$ is defined as in Theorem \ref{theorem: Holmstedt formulas}. For several properties concerning Sobolev spaces, we refer to \cite{Adams, Mazya, Brezis}.

\section{Mixed norm spaces}\label{mixed norm spaces}

Our goal in this section is to present some basic properties of mixed norm spaces we shall need for our work (in what follows and throughout the paper we shall assume $n\geq 2$.)

Let $k\in\{1,\ldots,n\}$. We write $\widehat{x}_{k}$  for the point in $I^{n-1}$ obtained from a given vector $x\in I^{n}$ by removing its $k\textnormal{th}$ coordinate. That is,
\begin{align*}
\widehat{x}_{k}=(x_{1},\ldots,x_{k-1},x_{k+1},\ldots,x_{n})\in I^{n-1}.
\end{align*}
Moreover, for any $f\in\mathcal{M}(I^{n}),$ we use the notation $f_{\widehat{x}_{k}}$  for the function  obtained from $f,$ with $\widehat{x}_{k}$ fixed. Observe   that, since $f$ is measurable,  $f_{\widehat{x}_{k}}$ is also measurable a.e. $\widehat{x}_{k}\in I^{n-1}$.

We now recall the Benedek-Panzone spaces, which were introduced in~\cite{Benedek-Panzone} for the case of $L^p$.  For further information on this topic see \cite{Buhvalov, Blozinski, Boccuto-Bukhvalov-Sambucini,Barza-Kaminska-Persson-Soria}.

\begin{definition} Let $k\in\{1,\ldots,n\}$.  Given two r.i.\ spaces $X(I^{n-1})$ and  $Y(I),$ the  Benedek-Panzone space $\mathcal{R}_{k}(X,Y)$ is defined  as the collection of all $f\in\mathcal{M}(I^{n})$ satisfying
\begin{align*}
\bigl\|f\bigr\|_{\mathcal{R}_{k}(X,Y)}=\bigl\|\psi_{k}(f,Y)\bigr\|_{X(I^{n-1})}<\infty,
\end{align*}
where $\psi_{k}(f,Y)(\widehat{x}_{k})=\bigl\|f(\widehat{x}_{k},\cdot)\bigr\|_{Y(I)}$.
\end{definition}

Buhvalov \cite{Buhvalov} and  Blozinski \cite{Blozinski} proved that   $\mathcal{R}_{k}(X,Y)$  is a Banach function space. Moreover, Boccuto, Bukhvalov, and Sambucini \cite{Boccuto-Bukhvalov-Sambucini} proved that $\mathcal{R}_{k}(X,Y)$ is an r.i.\ space, if and only if $X=Y=L^p$, for some $1\le p\le\infty$.

Now, we shall give the definition of the  mixed norm spaces, sometimes also called symmetric mixed norm spaces.

\begin{definition}\label{definition: mixed norm spaces}Given two r.i.\ spaces  $X(I^{n-1})$ and  $Y(I),$ the mixed norm space $\mathcal{R}(X,Y)$ is defined as
\begin{align*}
\mathcal{R}(X,Y)=\bigcap^{n}_{k=1}\mathcal{R}_{k}(X,Y).
\end{align*}
For each $f\in\mathcal{R}(X,Y),$ we set  $\bigl\|f\bigr\|_{\mathcal{R}(X,Y)}=\sum^{n}_{k=1}\bigl\|f\bigr\|_{\mathcal{R}_{k}(X,Y)}$.
\end{definition}
 
It is not difficult to verify that  $\mathcal{R}(X,Y)$ is a Banach function space. Since the pioneering works of Gagliardo~\cite{Gagliardo}, Nirenberg \cite{Nirenberg}, and Fournier \cite{Fournier}, many useful properties and generalizations of these spaces have been  studied, via different methods, by various authors, including Blei  \cite{Blei-Fournier}, Milman~\cite{Milman-kfunctional-mixed-norm}, Algervik and Kolyada \cite{Robert-Viktor}, and Kolyada \cite{Viktor2012, Viktor2013}.

 A thorough study of mixed norm spaces has been considered in  \cite{Clavero-Soria}.  In particular,  extending the mixed norm estimates due to Fournier \cite{Fournier} to more general r.i.\ spaces, we were able to  obtain a description of the smallest r.i.\ space  that is  continuously embedded into  a fixed  mixed norm space of the form $\mathcal{R}(X,L^{\infty})$ (see \cite[Theorem~5.6]{Clavero-Soria}):
  
 \begin{theorem}\label{teo: rango R(X,Loo)->Z} Let $X(I^{n-1})$ be an r.i. space . Then, the space $Z_{\mathcal{R}(X,L^{\infty})}(I^{n})$  of all $f\in \mathcal{M}(I^{n})$ such that
 \begin{align*}
 \|f\|_{Z_{\mathcal{R}(X,L^{\infty})}(I^{n})}= \bigl\|f^{*}(t^{n^{\prime}})\bigr\|_{\overline{X}(0,1)}<\infty,
 \end{align*}
 is an r.i.\ space.  Moreover, $Z_{\mathcal{R}(X,L^{\infty})}(I^{n})$ is the smallest r.i. space satisfying
 \begin{align}\label{eq: rango R(X,Loo)->Z}
\mathcal{R}(X,L^{\infty})\hookrightarrow Z_{\mathcal{R}(X,L^{\infty})}(I^{n})
\end{align}
 holds. 
 \end{theorem}
Furthermore, we   can see that  \eqref{eq: rango R(X,Loo)->Z} is a strict embedding when 
$$
Z(I^{n})\neq L^{\infty}(I^{n})=\mathcal{R}(L^{\infty},L^{\infty}).
$$

 \begin{theorem}\label{teo: mixta inclusion Z->R(X,Loo)} Let  $X(I^{n-1})$ and $Z(I^{n})$  be  r.i.\ spaces. Then, 
\begin{align*}
 Z(I^{n})\hookrightarrow\mathcal{R}(X,L^{\infty}) \Longleftrightarrow Z(I^{n})=L^{\infty}(I^{n}).
\end{align*}
\end{theorem}

To prove this, we first need to recall a result concerning embeddings between mixed norms (see \cite[Theorem~4.6]{Clavero-Soria}). 

\begin{theorem}\label{teo: mixta inclusion R(X1,Loo)->R(X2,Loo)} Let  $X_{1}(I^{n-1})$ and $X_{2}(I^{n-1})$ be  r.i.\ spaces. Then,
\begin{align*}
\mathcal{R}(X_{1},L^{\infty})\hookrightarrow \mathcal{R}(X_{2},L^{\infty})\Longleftrightarrow X_{1}(I^{n-1})\hookrightarrow X_{2}(I^{n-1}).
\end{align*}
\end{theorem}

\begin{proof}[Proof of Theorem~\ref{teo: mixta inclusion Z->R(X,Loo)}] 
 In view of Theorem~\ref{teo: mixta inclusion R(X1,Loo)->R(X2,Loo)}, it suffices to prove the necessary part of this result. We shall see that if 
\begin{align*}
\begin{array}{c}
L^{\infty}(I^{n})\hookrightarrow Z(I^{n}), \\[-9pt]    
{\neq}
\end{array} 
\end{align*}
then
\begin{align*}
Z(I^{n})\not\hookrightarrow\mathcal{R}(X,L^{\infty}).
\end{align*}
We may suppose, without loss of generality, that  $I=(-a,b),$ with $a,b\in\mathbb R_{+}$. Let $0<r<\min(a,b)$. Given any function $g\in Z(I^{n}),$ but $g\not\in L^{\infty}(I^{n}),$ we define 
 \begin{align*}
 f(x)=\begin{cases}
 g^{*}(2|x_{n}|),& \textnormal{if $(\widehat{x}_{n},x_{n})\in  I^{n-1}\times (-r,r),$}\\
 0,&\textnormal{otherwise.}
 \end{cases}
 \end{align*} 
Let us see that $f\in Z(I^{n})$ and $f\notin \mathcal{R}(X,L^{\infty})$. In fact, using  \eqref{eq: radial function}, we get
 \begin{align*}
f^{*}(t)=\begin{cases}
g^{*}(t),& \textnormal{if $0\leq t<\min\bigl(2r,\lambda_{g}(0)\bigr),$}\\
0,&\textnormal{otherwise.}\end{cases}
 \end{align*}
 Hence, our assumption on $g$ ensures that 
 \begin{align*}
 \big\|f\big\|_{Z(I^{n})}\leq \big\|g\big\|_{Z(I^{n})}<\infty.
 \end{align*}
  On the other hand, for any $\widehat{x}_{n}\in I^{n-1},$ it holds that 
 \begin{align*}
 \psi_{n}(f,L^{\infty})(\widehat{x}_{n})=\left\|g\right\|_{L^{\infty}(I)}=\infty.
 \end{align*}
 Hence $f\not\in\mathcal{R}_{n}(X,L^{\infty})$ and  the proof is complete.
\end{proof}

Taking into account Theorem~\ref{teo: mixta inclusion Z->R(X,Loo)}, it is immediate to see that a mixed norm space $\mathcal{R}(X,L^{\infty})$ is an r.i.\ space if and only if $\mathcal{R}(X,L^{\infty})=L^{\infty}(I^{n}),$ which is equivalent to $X(I^{n-1})=L^{\infty}(I^{n-1})$.
 
We end this section by recalling the expression of the $K$-functional for the couple of mixed norm spaces $(\mathcal{R}(X,L^{\infty}),L^{\infty})$ given in \cite{Clavero-Soria}. 

\begin{theorem}\label{theorem: K-functional of R(X,Loo) and Loo} Let $X(I^{n-1})$ be an r.i.\ space and let   $f\in  \mathcal{R}(X,L^{\infty})+L^{\infty}(I^{n})$. Then, 
\begin{align*}
K(f,\varphi_{X}(t), \mathcal{R}(X,L^{\infty}),L^{\infty})\approx\sum^{n}_{k=1}\left\|\psi^{*}_{k}(f,L^{\infty})\chi_{(0,t)}\right\|_{\overline{X}(0,)},\ \ t>0,
\end{align*}
where $\varphi_{X}(t)$ is the fundamental function of $X(I^{n-1})$ defined in \eqref{eq: fundamental function}.
 
\end{theorem}

\section{Sobolev embeddings in mixed norm spaces}\label{sec: Sobolev embedding in mixed norm spaces}
Our aim in this section is to study the Sobolev embedding of the form
\begin{align}\label{eq: intro W1Z->R(X,Loo)}
W^{1}Z(I^{n})\hookrightarrow \mathcal{R}(X,L^{\infty}),
\end{align}
extending the classical estimate \eqref{eq: Gagliardo-Niremberg} proved by Gagliardo \cite{Gagliardo} and Nirenberg \cite{Nirenberg}.
 In particular, we are interested in the following problems:
\begin{enumerate}[(i)]
	\item We would like to find the smallest  space of the form  $\mathcal{R}(X,L^{\infty})$ in \eqref{eq: intro W1Z->R(X,Loo)}, for a given r.i.\  $Z(I^{n})$.
	\item On the other hand, given a fixed range space $\mathcal{R}(X,L^{\infty}),$ we would like to provide a characterization of the largest r.i.\ domain space  satisfying  \eqref{eq: intro W1Z->R(X,Loo)}. 
	\end{enumerate}

\subsection{Necessary and sufficient conditions}Now, our main purpose is to find  necessary and sufficient conditions
on $X(I^{n-1})$  and $Z(I^{n})$ under which we have the embedding \eqref{eq: intro W1Z->R(X,Loo)}. 

For this, we shall establish the equivalence between \eqref{eq: intro W1Z->R(X,Loo)} and the  boundedness of a Hardy type operator, via an argument used by Kerman and Pick \cite{Kerman-Pick} to characterize Sobolev embeddings in r.i.\ spaces. Then,  this relation will be   a key
tool in determining the largest r.i.\  space and the smallest mixed norm space for \eqref{eq: intro W1Z->R(X,Loo)}.

Let us start with an auxiliary lemma. The proof, based on   a classical interpolation result due to Calder\'on (see \cite[Theorem~III.2.12]{Bennett}), follows the scheme of \cite[Lemma~4.1]{Cianchi}, so we do not include it here.

\begin{lemma}\label{lema: operador} Let $\beta>-1$  and let $Y(0,1)$ be an r.i.\ space. Then, 
\begin{align*}
\bigg\|\int^{1}_{t}s^{\beta}f(s)ds\bigg\|_{Y(0,1)}\lesssim  \|f\|_{Y(0,1)}, \ f\in Y(0,1).
\end{align*}
\end{lemma}

\begin{theorem}\label{theorem: reduccion medida finita} Let $X(I^{n-1})$ and $Z(I^{n})$ be  r.i.\ spaces.
 Then,  the following statements are equivalent: 
 \begin{enumerate}[(i)]\vspace{0.5cm}
 \item $
 W^{1}Z(I^{n})\hookrightarrow \mathcal{R}(X,L^{\infty});$ \vspace{0.5cm}
 \item  $
 \bigg\|\displaystyle\int^{1}_{t^{n^{\prime}}}f^{*}(s) s^{-1/n^{\prime}}ds\bigg\|_{\overline{X}(0,1)}\lesssim  \|f^{*}\|_{\overline{Z}(0,1)},\ \  f\in Z(I^{n})$.
  \end{enumerate}
  \end{theorem}

 \begin{proof} 
 First we prove that $(i)  \Rightarrow  (ii)$. As in the proof of Theorem~\ref{teo: mixta inclusion Z->R(X,Loo)}, we may assume that $I=(-a,b),$ with $a,b\in\mathbb R_{+}$ and $0<r<\min(a,b)$. Given any function  $f\in Z(I^{n}),$ with $\lambda_{f}(0)\leq \omega^{n^{\prime}}_{n-1}r^{n},$  we define 
 \begin{align*}
 u(x)=\begin{cases}
 \displaystyle{\int^{\omega^{n^{\prime}}_{n-1}r^{n}}_{\omega^{n^{\prime}}_{n-1}|x|^{n}}s^{-1/n^{\prime}}f^{*}(s)ds},&\textnormal{if $x\in B_{n}(0,r),$}\\
 0,&\textnormal{otherwise.}
 \end{cases}
 \end{align*}
Then,  by  \eqref{eq: radial function} and
the boundedness of the dilation operator in r.i.\ spaces, we get  
 \begin{align*}
\|u\|_{Z(I^{n})}&\lesssim  \bigg\|\int^{I^{n}}_{t}s^{-1/n^{\prime}}f^{*}(s)ds\bigg\|_{\overline{Z}(0,1)},
\end{align*}
and so, Lemma~\ref{lema: operador} gives
\begin{align}\label{eq: 3W1Z->R(X,Loo)}
\|u\|_{Z(I^{n})}\lesssim \|f^{*}\|_{\overline{Z}(0,1)}.
\end{align}
 On the other hand, we have
  \begin{align*}
 |\nabla u(x)|\approx f^{*}( \omega^{n^{\prime}}_{n-1}|x|^{n}),\ \ \textnormal{a.e. $x\in B_{n}(0,r),$}
 \end{align*}
 and $|\nabla u(x)|=0$ otherwise. So, using again  the boundedness of the dilation operator in r.i.\ spaces, we get
\begin{align}\label{eq: 4W1Z->R(X,Loo)}
\||\nabla u|\|_{Z(I^{n})}&\lesssim \|f^{*}\|_{\overline{Z}(0,1)}.
\end{align}
By hypothesis  $f\in Z(I^{n}),$ so inequalities \eqref{eq: 3W1Z->R(X,Loo)} and \eqref{eq: 4W1Z->R(X,Loo)} imply that  $u\in W^{1}Z(I^{n})$ and
 \begin{align*}
 \|u\|_{W^{1}Z(I^{n})}=\|u\|_{Z(I^{n})}+\||\nabla u|\|_{Z(I^{n})}\lesssim \|f^{*}\|_{\overline{Z}(0,1)}.
 \end{align*}
 Therefore, using  
 $W^{1}Z(I^{n})\hookrightarrow \mathcal{R}(X,L^{\infty}),$ we obtain
 \begin{align}\label{eq: 5W1Z->R(X,Loo)}
 \big\|u\big\|_{\mathcal{R}(X,L^{\infty})}\lesssim \big\|f^{*}\big\|_{\overline{Z}(0,1)}.
 \end{align}
 Now, let us compute $\|u\|_{\mathcal{R}(X,L^{\infty})}$. For this, we fix any $k\in \{1,\ldots,n\}.$ Then,  we have
 \begin{align*}
\psi_{k}(u,L^{\infty})(\widehat{x}_{k})=\begin{cases}\displaystyle{\int^{\omega^{n^{\prime}}_{n-1}r^{n}}_{\omega^{n^{\prime}}_{n-1}|\widehat{x}_{k}|^{n}}s^{-1/n^{\prime}}f^{*}(s)ds},&\textnormal{if $\widehat{x}_{k}\in B_{n-1}(0,r),$}\\
0,&\textnormal{otherwise.}
\end{cases}
 \end{align*}
 As a consequence,  using again  \eqref{eq: radial function}, we get
 \begin{align*}
  \|u\|_{\mathcal{R}(X,L^{\infty})}=\sum^{n}_{k=1}\|\psi^{*}_{k}(u,L^{\infty})\|_{\overline{X}(0,1)}=n\bigg\|\int^{1}_{t^{n^{\prime}}}s^{-1/n^{\prime}}f^{*}(s)ds\bigg\|_{\overline{X}(0,1)}.
 \end{align*}
 Thus, using \eqref{eq: 5W1Z->R(X,Loo)}, we obtain 
 \begin{align}\label{eq: 6W1Z->R(X,Loo)}
 \biggl\|\int^{1}_{t^{n^{\prime}}}s^{-1/n^{\prime}}f^{*}(s)ds\biggr\|_{\overline{X}(0,1)}\lesssim  \|f\|_{\overline{Z}(0,1)}.
 \end{align}
 This proves $(ii)$, for any  $f\in Z(I^{n}),$ with $\lambda_{f}(0)\leq \omega^{n^{\prime}}_{n-1}r^{n}.$ Now, let us consider any $f\in Z(I^{n}).$ We define
\begin{align*}
f_{1}(x)=\max\bigl(|f(x)|-f^{*}(\omega^{n^{\prime}}_{n-1}r^{n}),0\bigr)\,\textnormal{sgn}f(x),
\end{align*}
and 
\begin{align*}
f_{2}(x)=\min\bigl(|f(x)|,f^{*}(\omega^{n^{\prime}}_{n-1}r^{n})\bigr)\,\textnormal{sgn}f(x).
\end{align*}
We observe that $f=f_{1}+f_{2},$ 
\begin{align}\label{eq: 7W1Z->R(X,Loo)}
f^{*}_{1}(t)=\begin{cases}
f^{*}(t)-f^{*}(\omega^{n^{\prime}}_{n-1}r^{n}),& 0\leq t<\lambda_{f}(f^{*}(\omega^{n^{\prime}}_{n-1}r^{n})),\\
0,&\textnormal{otherwise},
\end{cases}
\end{align}
and 
\begin{align}\label{eq: 8W1Z->R(X,Loo)}
f^{*}_{2}(t)=\begin{cases}
f^{*}(\omega^{n^{\prime}}_{n-1}r^{n}), & 0\leq t<\lambda_{f}(f^{*}(\omega^{n^{\prime}}_{n-1}r^{n})),\\
f^{*}(t),&\textnormal{otherwise}.
\end{cases}
\end{align}
So, combining \eqref{eq: 7W1Z->R(X,Loo)} and \eqref{eq: 8W1Z->R(X,Loo)}, we have
$f^{*}=f^{*}_{1}+f^{*}_{2}.$ Using now  inequality \eqref{eq: 6W1Z->R(X,Loo)}, with $f$ replaced by $f_{1},$  we get
\begin{align}\label{eq: 9W1Z->R(X,Loo)}
\biggl\|\int^{1}_{t^{n^{\prime}}}s^{-1/n^{\prime}}f^{*}_{1}(s)ds\biggr\|_{\overline{X}(0,1)}\lesssim  \big\|f^{*}\big\|_{Z(0,1)}.
\end{align}
 On the other hand, by H\"older's inequality, we have
 \begin{align}\label{eq: 10W1Z->R(X,Loo)}
 \biggl\|\int^{1}_{t^{n^{\prime}}}s^{-1/n^{\prime}}f^{*}_{2}(s)ds\biggr\|_{\overline{X}(0,1)}\lesssim f^{**}(\omega^{n^{\prime}}_{n-1}r^{n})\lesssim \|f\|_{Z(0,1)}. 
 \end{align}
 As a consequence, using \eqref{eq: 9W1Z->R(X,Loo)} and \eqref{eq: 10W1Z->R(X,Loo)},  we get
\begin{align*}
\biggl\|\int^{1}_{t^{n^{\prime}}}s^{-1/n^{\prime}}f^{*}(s)ds\biggr\|_{\overline{X}(0,1)}&=\biggl\|\int^{1}_{t^{n^{\prime}}}s^{-1/n^{\prime}}\big(f^{*}_{1}(s)+f^{*}_{2}(s)\big)ds\biggr\|_{\overline{X}(0,1)}\\
&\leq \biggl\|\int^{1}_{t^{n^{\prime}}}s^{-1/n^{\prime}}f^{*}_{1}(s)ds\biggr\|_{\overline{X}(0,1)}\\
&\quad+\biggl\|\int^{1}_{t^{n^{\prime}}}s^{-1/n^{\prime}}f^{*}_{2}(s)ds\biggr\|_{\overline{X}(0,1)}\lesssim \|f^{*}\|_{Z(0,1)},
\end{align*}
 which is $(ii).$
 
 Conversely, let us suppose that $(ii)$ holds. We fix any  $f\in W^{1}Z(I^{n})$.  Combining the classical embedding on Lorentz spaces (cf. e.g. \cite{Oneil-1963, Talenti})
\begin{align*}
W^{1}L^{n,1}(I^{n})\hookrightarrow \mathcal{R}(L^{\infty},L^{\infty})=L^{\infty}(I^{n}),
\end{align*} 
with  Gagliardo-Nirenberg embedding \eqref{eq: Gagliardo-Niremberg}, and then applying Theorem~ \ref{theorem: interpolation operator}, we get 
\begin{align}\label{eq: 11W1Z->R(X,Loo)}
K(f,t; \mathcal{R}(L^{1},L^{\infty}), L^{\infty})\lesssim K(f,Ct; W^{1}L^{1}, W^{1}L^{n,1}), \ \ 0<t<1.
\end{align}
We have, by Theorem~\ref{theorem: K-functional of R(X,Loo) and Loo}, 
\begin{align}\label{eq: 12W1Z->R(X,Loo)}
K(f,t; \mathcal{R}(L^{1},L^{\infty}), L^{\infty})\approx \sum^{n}_{k=1}\int^{t}_{0}\psi^{*}_{k}(f,L^{\infty})(s)ds.
\end{align}
Moreover, using now \eqref{eq: Holmstedt sobolev}, we get
\begin{align*}
K(f,Ct; W^{1}L^{1}, W^{1}L^{n,1})&\approx \int^{(Ct)^{n^{\prime}}}_{0}|D^{1}f|^{*}(s)ds+ Ct \int^{1}_{(Ct)^{n^{\prime}}}s^{-1/n^{\prime}}|D^{1}f|^{*}(s)ds\\
&\approx \int^{(Ct)^{n^{\prime}}}_{0}s^{-1/n}\bigg(\int^{1}_{s}y^{-1/n^{\prime}}|D^{1}f|^{*}(y)dy\bigg)ds.
\end{align*}
So, by a change of variables, we obtain
 \begin{align}\label{eq: 13W1Z->R(X,Loo)}
 K(f,Ct; W^{1}L^{1}, W^{1}L^{n,1})&\approx \int^{t}_{0}\bigg(\int^{1}_{Cs^{n^{\prime}}}y^{-1/n^{\prime}}|D^{1}f|^{*}(y)dy\bigg)ds.
 \end{align}
  Therefore, taking into account  \eqref{eq: 11W1Z->R(X,Loo)},  \eqref{eq: 12W1Z->R(X,Loo)}, and \eqref{eq: 13W1Z->R(X,Loo)},  we obtain
\begin{align*}
\int^{t}_{0}\psi^{*}_{k}(f,L^{\infty})(s)ds\lesssim \int^{t}_{0}\bigg(\int^{1}_{Cs^{n^{\prime}}}y^{-1/n^{\prime}}|D^{1}f|^{*}(y)dy\bigg)ds,\ \ k\in \{1,\ldots,n\}.
\end{align*}
So,  using Hardy-Littlewood-P\'olya Principle, the boundedness of the dilation operator in r.i.\ spaces and $(ii),$ we get
\begin{align*}
\|f\|_{\mathcal{R}_{k}(X,L^{\infty})}&\lesssim \bigg\|\int^{1}_{s^{n^{\prime}}} y^{-1/n^{\prime}}|D^{1} f|^{*}(y)dy \bigg\|_{\overline{X}(0,1)}\lesssim \||D^{1}f|^{*}\|_{\overline{Z}(0,1)}\approx \|f\|_{W^{1}Z(I^{n})},
\end{align*}
for any $ k\in \{1,\ldots,n\}$, from which $(i)$ follows. 
 \end{proof}
 
 \begin{remark}\label{remark: extra condition}\textnormal{Using a duality argument, we observe that the statements proved in Theorem~\ref{theorem: reduccion medida finita}  are also equivalent to the following  additional condition:
 \begin{align*}
 \sup_{\|f\|_{X^{\prime}(I^{n-1})}\leq 1}\|f^{**}(t^{1/n^{\prime}})\|_{\overline{Z}^{\prime}(0,1)}=\sup_{\|g\|_{Z(I^{n})}\leq 1}\bigg\|\int^{1}_{s^{n^{\prime}}}g^{*}(t)t^{-1/n^{\prime}}dt\bigg\|_{\overline{X}(0,1)}<\infty.
 \end{align*}
 In fact, we fix any $f\in X^{\prime}(I^{n-1}),$ with  $\|f\|_{X^{\prime}(I^{n-1})}\leq 1.$ Then, by Fubini's theorem and H\"older's inequality, we get
  \begin{align*}
 \|f^{**}(t^{1/n^{\prime}})\|_{\overline{Z}^{\prime}(0,1)}&=
 \sup_{\|g\|_{Z(I^{n})}\leq 1}\int^{1}_{0}f^{*}(t)\Big(\int^{1}_{t^{n^{\prime}}}s^{-1/n^{\prime}}g^{*}(s)ds\Big)dt\\
 &\leq \sup_{\|g\|_{Z(I^{n})}\leq 1} \|f\|_{X^{\prime}(I^{n-1})}\bigg\|\int^{1}_{t^{n^{\prime}}}s^{-1/n^{\prime}}g^{*}(s)ds\bigg\|_{\overline{X}(0,1)}\\
 &\leq \sup_{\|g\|_{Z(I^{n})}\leq 1} \bigg\|\int^{1}_{t^{n^{\prime}}}s^{-1/n^{\prime}}g^{*}(s)ds\bigg\|_{\overline{X}(0,1)}.
 \end{align*} 
Therefore, we conclude that
 \begin{align*}
  \sup_{\|f\|_{X^{\prime}(I^{n-1})}\leq 1}\|f^{**}(t^{1/n^{\prime}})\|_{\overline{Z}^{\prime}(0,1)}\leq \sup_{\|g\|_{Z(I^{n})}\leq 1}\bigg\|\int^{1}_{t^{n^{\prime}}}s^{-1/n^{\prime}}g^{*}(s)ds\bigg\|_{\overline{X}(0,1)}.
 \end{align*}
Applying the same arguments as before, we obtain the converse inequality.}
 \end{remark}
 
 \subsection{Characterization of the optimal range}
 Now, we fix  an r.i.\ space  $Z(I^{n}).$  We shall provide a description of the smallest space of the form $\mathcal{R}(X,L^{\infty})$ satisfying  
\begin{align*}
W^{1}Z(I^{n})\hookrightarrow \mathcal{R}(X,L^{\infty}).
\end{align*}

It is important to note that Theorem~\ref{theorem: reduccion medida finita}, together with Remark~\ref{remark: extra condition}, relate this problem with that of finding the largest r.i.\ space $Y(I^{n-1})$  such that
\begin{align*}
H^{\prime}: Y(I^{n-1})\rightarrow \overline{Z}^{\prime}(0,1)
\end{align*} 
is bounded, where $H^{\prime}$ is the conjugate Hardy type operator:
\begin{align}\label{hardytypeoper}
H^{\prime}f(t)=f^{**}(t^{1/n^{\prime}}), \ \ f\in\mathcal{M}(I^{n-1}).
\end{align}
Hence, it is natural to introduce a new space, denoted by $Y(I^{n-1}),$ consisting of all $f\in \mathcal{M}(I^{n-1})$ such that
\begin{align}\label{eq: 1W1Z->R(X,Loo) norm}
\|f\|_{Y(I^{n-1})}=\|f^{**}(t^{-1/n^{\prime}})\|_{\overline{Z}^{\prime}(0,1)}<\infty. 
\end{align}
It is not difficult to verify that  $Y(I^{n-1})$ is an r.i.\ space equipped with the norm $\|\cdot\|_{Y(I^{n-1})}.$ Hence, using \eqref{hardytypeoper}, \eqref{eq: 1W1Z->R(X,Loo) norm}, and a duality argument, we have that its associate space $Y^{\prime}(I^{n-1})$ verifies that 
\begin{align}\label{eq: 2W1Z->R(X,Loo) rango}
H:\overline{Z}(0,1) \rightarrow \overline{Y}^{\prime}(I^{n-1}),
\end{align} 
where $H$ is the  Hardy type operator:
\begin{align}\label{eq: Hardy type operator}
Hf(t)=\int^{1}_{t^{n^{\prime}}}s^{-1/n^{\prime}}f^{*}(s)ds, \ \ f\in\mathcal{M}(I^{n}).
\end{align}

In order to clarify the notation used later, note  that if we denote by
\begin{align}\label{eq: 1W1Z->R(X,Loo) rango}
X_{W^{1}Z,L^{\infty}}(I^{n-1}):=Y^{\prime}(I^{n-1}),
\end{align}
then, \cite[Theorem I.2.7]{Bennett} implies that
\begin{align*}
Y(I^{n-1})=(Y^{\prime})^{\prime}(I^{n-1})=X^{\prime}_{W^{1}Z,L^{\infty}}(I^{n-1}).
\end{align*}

\begin{theorem}\label{teo: W1Z->R(X,Loo) rango} Let $Z(I^{n})$ be an r.i.\ space and let  $X_{W^{1}Z,L^{\infty}}(I^{n-1})$ be the r.i.\ space defined in \eqref{eq: 1W1Z->R(X,Loo) rango}. Then, the Sobolev embedding
\begin{align} \label{eq: 3W1Z->R(X,Loo) rango}
W^{1}Z(I^{n})\hookrightarrow \mathcal{R}(X_{W^{1}Z,L^{\infty}},L^{\infty}),
\end{align}
holds. Moreover, $\mathcal{R}(X_{W^{1}Z,L^{\infty}},L^{\infty})$ is the smallest space of the form $\mathcal{R}(X,L^{\infty})$ that verifies \eqref{eq: 3W1Z->R(X,Loo) rango}. 
\end{theorem}

\begin{proof}The embedding  \eqref{eq: 3W1Z->R(X,Loo) rango} follows directly from Theorem~\ref{theorem: reduccion medida finita} together with \eqref{eq: 2W1Z->R(X,Loo) rango}. Thus, to complete the proof, it only remains  to see that    $\mathcal{R}(X_{W^{1}Z,L^{\infty}},L^{\infty})$ is the smallest space of the form $\mathcal{R}(X,L^{\infty})$ satisfying \eqref{eq: 3W1Z->R(X,Loo) rango}. Hence, we shall see that if a mixed norm space $\mathcal{R}(X,L^{\infty})$ verifies
\begin{align}\label{eq: 4W1Z->R(X,Loo) rango}
W^{1}Z(I^{n})\hookrightarrow \mathcal{R}(X,L^{\infty}),
\end{align}
then
\begin{align*}
\mathcal{R}(X_{W^{1}Z,L^{\infty}},L^{\infty})\hookrightarrow \mathcal{R}(X,L^{\infty}).
\end{align*}
 We fix any $g\in X^{\prime}(I^{n-1}).$ Then, combining  \eqref{eq: 4W1Z->R(X,Loo) rango} with Remark~\ref{remark: extra condition}, we get 
\begin{align*}
\|g^{**}(t^{1/n^{\prime}})\|_{\overline{Z}^{\prime}(0,1)} \lesssim  \|g\|_{X^{\prime}(I^{n-1})}.
 \end{align*}
Therefore, using now \eqref{eq: 1W1Z->R(X,Loo) norm}, we obtain
 \begin{align*}
X^{\prime}(I^{n-1})\hookrightarrow X^{\prime}_{W^{1}Z,L^{\infty}}(I^{n-1}).
 \end{align*}
 As a consequence,  \cite[Proposition~I.2.10]{Bennett} and Theorem~\ref{teo: mixta inclusion R(X1,Loo)->R(X2,Loo)} imply that \eqref{eq: 3W1Z->R(X,Loo) rango} holds, as we wanted to show.  \end{proof}

Now, we shall present some applications of Theorem~\ref{teo: W1Z->R(X,Loo) rango}. In particular, we shall see that \eqref{eq: Gagliardo-Niremberg} cannot be improved within the class of  spaces of the form $\mathcal{R}(X,L^{\infty}).$ This should be understood as follows: if we replace the range space  in
\begin{align*}
W^{1}L^{1}(I^{n})\hookrightarrow \mathcal{R}(L^{1},L^{\infty}),
\end{align*}
by a smaller mixed norm space, say $\mathcal{R}(X,L^{\infty})$, then the resulting embedding
\begin{align*}
W^{1}L^{1}(I^{n})\hookrightarrow \mathcal{R}(X,L^{\infty})
\end{align*}
cannot longer be true.

\begin{corollary}\label{corollary: W1Lp->R(X,Loo) rango}
Let $1\leq p<n.$ Then, the mixed norm space $\mathcal{R}(L^{p(n-1)/(n-p),p},L^{\infty})$ is the smallest space of the form $\mathcal{R}(X,L^{\infty})$ satisfying 
\begin{align*}
W^{1}L^{p}(I^{n})\hookrightarrow \mathcal{R}(L^{p(n-1)/(n-p),p},L^{\infty}).
\end{align*}
\end{corollary}

\begin{proof} 
We prove this result only when $1<p<n$ ($p=1$ is easier). Theorem~\ref{teo: W1Z->R(X,Loo) rango},  with $Z(I^{n})$ replaced by $L^{p}(I^{n}),$ gives 
\begin{align*}
\|f\|_{X^{\prime}_{W^{1}L^{p},L^{\infty}}(I^{n-1})}&=\|f^{**}(t^{1/n^{\prime}})\|_{L^{p^{\prime}}(0,1)}\\
&\approx \bigg(\int^{1}_{0}t^{-(n-p)/((n-1)(p-1))-1}\bigg(\int^{t}_{0}f^{*}(s)ds\bigg)^{p^{\prime}} dt\bigg)^{1/p^{\prime}}. 
\end{align*}
Since $1<p<n,$  we may apply Hardy's inequalities  \cite[Lemma~III.3.9]{Bennett} to  get   
\begin{align}\label{eq: 1W1Lp->R(X,Loo) rango}
\|f\|_{X^{\prime}_{W^{1}L^{p},L^{\infty}}(I^{n-1})}&\lesssim \|f\|_{L^{p^{\prime}(n-1)/n,p^{\prime}}(I^{n-1})}. 
\end{align}
On the other hand,  we have
\begin{align}\label{eq: 2W1Lp->R(X,Loo) rango}
\|f\|_{X^{\prime}_{W^{1}L^{p},L^{\infty}}(I^{n-1})} \gtrsim \|f\|_{L^{p^{\prime}(n-1)/n,p^{\prime}}(I^{n-1})}.
\end{align}
As a consequence, combining \eqref{eq: 1W1Lp->R(X,Loo) rango} and \eqref{eq: 2W1Lp->R(X,Loo) rango}, we get
\begin{align*}
X^{\prime}_{W^{1}L^{p},L^{\infty}}(I^{n-1})=L^{p^{\prime}(n-1)/n ,p^{\prime}}(I^{n-1}),
\end{align*}
and hence,  using \cite[Theorem~IV.4.7]{Bennett}, we have that 
\begin{align*}
X_{W^{1}L^{p},L^{\infty}}(I^{n-1})=L^{p(n-1)/(n-p),p}(I^{n-1}),
\end{align*}
from which the result follows.
\end{proof}

Now, we shall apply Theorem~\ref{teo: W1Z->R(X,Loo) rango} to the  so-called  limiting or critical
case of the classical Sobolev embedding.

\begin{corollary}\label{corollary: W1Ln->R(X,Loo) rango}
The mixed norm space $\mathcal{R}(L^{\infty,n;-1},L^{\infty})$  is the smallest space of the form $\mathcal{R}(X,L^{\infty})$ satisfying 
 \begin{align}\label{eq: W1Ln->R(X,Loo) rango} 
W^{1}L^{n}(I^{n})\hookrightarrow \mathcal{R}(L^{\infty,n;-1},L^{\infty}).\end{align}
\end{corollary}

The  proof will be an immediate consequence of Theorem~\ref{teo: W1Z->R(X,Loo) rango} and the
following  result given in \cite{Edmunds-Kerman-Pick}.

\begin{theorem}\label{theorem: associate norm}Let $1<p<\infty$ and let $v$ be a weight on $(0,1)$ satisfying the following properties:
\begin{enumerate}[ (i)]
\item $\int^{1}_{0}v(t)dt<\infty;$ \vspace{0.2 cm}
\item $\int^{1}_{0}t^{-p}v(t)^{p}dt=\infty;$ \vspace{0.2 cm}
\item $\int^{r}_{0}v(t)^{p}dt\lesssim r^{p}\big(1+\int^{1}_{r}t^{-p}v(t)^{p}dt\big),$ $0<r<1.$\vspace{0.2 cm}
\end{enumerate}
Then, the r.i.\ norm defined as
\begin{align*}
\|f\|_{X(0,1)}=\|v(t)f^{**}(t)\|_{L^{p}(0,1)},  \  \ f\in\mathcal{M}(0,1)
\end{align*}
has associate norm
\begin{align*}
\|g\|_{X^{\prime}(0,1)}=\|w(t)g^{*}(t)\|_{L^{p^{\prime}}(0,1)},  \  \ g\in\mathcal{M}(0,1),
\end{align*}
where
\begin{align*}
w(t)^{p^{\prime}}=\frac{d}{dt}\bigg[\bigg(1+\int^{1}_{t}s^{-p}v(s)^{p}ds\bigg)^{1-p^{\prime}}\bigg], \  \ 0<t<1.
\end{align*}
\end{theorem}

\begin{proof}[Proof of Corollary~\ref{corollary: W1Ln->R(X,Loo) rango}] By Theorem~\ref{teo: W1Z->R(X,Loo) rango}, with $Z(I^{n})=L^{n}(I^{n}),$ we get
\begin{align*}
\|f\|_{X^{\prime}_{W^{1}L^{n},L^{\infty}}(I^{n-1})}&=\|f^{**}(t^{1/n^{\prime}})\|_{L^{n^{\prime}}(I^{n})}.\end{align*} 
Consequently, Theorem~\ref{theorem: associate norm} implies that
\begin{align*}
\|f\|_{X_{W^{1}L^{n},L^{\infty}}(I^{n-1})}&=\big\|t^{-1}\log\big(e/t\big)f^{*}(s)ds\big\|_{L^{n}(I^{n-1})},\end{align*} 
 from which it follows that 
\begin{align*}
\mathcal{R}(X_{W^{1}L^{n},L^{\infty}}L^{\infty})=\mathcal{R}(L^{\infty,n;-1},L^{\infty}),
\end{align*}
as we wanted to prove.
\end{proof}

\subsection{Characterization of the optimal domain} We    now focus on the problem of determining  the largest r.i.\ domain space  satisfying  \eqref{eq: intro W1Z->R(X,Loo)} for a fixed range space $\mathcal{R}(X,L^{\infty}).$
Observe that the equivalences proved in  Theorem~\ref{theorem: reduccion medida finita}  suggest that  in order to solve this problem, we should find  the largest r.i.\ space $Z(I^{n})$  such that 
\begin{align*}
H: Z(I^{n})\rightarrow \overline{X}(0,1)
\end{align*} 
is bounded, where $H$ is the Hardy type operator defined in \eqref{eq: Hardy type operator}. Hence, it is natural to consider a new space, denoted by $Z_{\mathcal{R}(X,L^{\infty})}(I^{n})$, consisting of all $f\in\mathcal{M}(I^{n})$ for which
\begin{align}\label{eq: norma W1Z->R(X,Loo) dominio}
\|f\|_{Z_{\mathcal{R}(X,L^{\infty})}}= \bigg\|\int^{1}_{t^{n^{\prime}}}f^{**}(s)s^{-1/n^{\prime}}ds\bigg\|_{\overline{X}(0,1)}<\infty.
\end{align}
 It is not difficult to verify that  $Z_{\mathcal{R}(X,L^{\infty})}(I^{n})$ is an r.i.\ space equipped with the norm $\|\cdot\|_{Z_{\mathcal{R}(X,L^{\infty})}(I^{n})}.$   The next  lemma will be needed later. Its  proof follows the same arguments used in \cite[Theorem~4.4]{Edmunds-Kerman-Pick}, with small modifications, and hence we will omit it.

\begin{lemma}\label{lema: norma dominio equivalente}  Let $X(I^{n-1})$ be an r.i space, with $\overline{\alpha}_{X}<1.$  Then,
 \begin{align*}
 \|f\|_{Z_{\mathcal{R}(X,L^{\infty})}(I^{n})}\approx \bigg\|\int^{1}_{t^{n^{\prime}}}f^{*}(s)s^{-1/n^{\prime}}ds\bigg\|_{\overline{X}(0,1)},\ \ f\in \mathcal{M}(I^{n}).
 \end{align*}
 \end{lemma}

\begin{theorem}\label{teo: W1Z->R(X,Loo) dominio} Let $X(I^{n-1})$ be an r.i space, $\overline{\alpha}_{X}<1,$ and let  $Z_{\mathcal{R}(X, L^{\infty})}(I^{n})$ be the r.i.\ space defined in \eqref{eq: norma W1Z->R(X,Loo) dominio}. Then, the Sobolev embedding
 \begin{align}\label{eq: 1W1Z->R(X,Loo) dominio}
 W^{1}Z_{\mathcal{R}(X,L^{\infty})}(I^{n})\hookrightarrow \mathcal{R}(X,L^{\infty})
 \end{align}
 holds. Moreover, $Z_{\mathcal{R}(X,L^{\infty})}(I^{n})$ is the largest domain space for which \eqref{eq: 1W1Z->R(X,Loo) dominio} holds.
 \end{theorem}
 
 \begin{proof} 
  Theorem~\ref{theorem: reduccion medida finita}  ensures us that \eqref{eq: 1W1Z->R(X,Loo) dominio} holds. Hence, to complete the proof,  it only remains to show that   $Z_{\mathcal{R}(X,L^{\infty})}(I^{n})$ is the largest r.i.\ space satisfying \eqref{eq: 1W1Z->R(X,Loo) dominio}.  We shall see that if another space, namely $Z(I^{n}),$ verifies 
 \begin{align*}
 W^{1}Z(I^{n})\hookrightarrow \mathcal{R}(X,L^{\infty}),
 \end{align*} 
 then 
 \begin{align*}
 Z(I^{n})\hookrightarrow Z_{\mathcal{R}(X,L^{\infty})}(I^{n}).
 \end{align*} 
 We fix any $f\in Z(I^{n}).$ Then, using again   Theorem~\ref{theorem: reduccion medida finita} , we get
 \begin{align}\label{eq: 2W1Z->R(X,Loo) dominio}
\bigg\|\int^{1}_{t^{n^{\prime}}}f^{*}(s)s^{-1/n^{\prime}}ds\bigg\|_{\overline{X}(0,1)}\lesssim \|f^{*}\|_{\overline{Z}(0,1)}.
 \end{align}
  But, by Lemma~\ref{lema: norma dominio equivalente}, we have
 \begin{align*} 
 \|f\|_{Z_{\mathcal{R}(X,L^{\infty})}(I^{n})}\approx\bigg\|\int^{1}_{t^{n^{\prime}}}s^{-1/n^{\prime}}f^{*}(s)ds\bigg\|_{\overline{X}(0,1)},
 \end{align*}
 and hence, using \eqref{eq: 2W1Z->R(X,Loo) dominio}, the result follows.
  \end{proof}
 
Now, we shall present some applications   of Theorem~\ref{teo: W1Z->R(X,Loo) dominio}. 

\begin{corollary} Let $1<p<n.$ Then, the Lebesgue space $L^{p}(I^{n})$ is the largest r.i.\ space satisfying
\begin{align*}
W^{1}L^{p}(I^{n})\hookrightarrow \mathcal{R}(L^{p(n-1)/(n-p),p},L^{\infty}).
\end{align*}
\end{corollary}

\begin{proof} 
If $X(I^{n-1})=L^{p(n-1)/(n-p),p}(I^{n-1})$ then, by  Theorem~\ref{teo: W1Z->R(X,Loo) dominio}, we obtain that
\begin{align*}
\|f\|_{Z_{\mathcal{R}(L^{p(n-1)/(n-p),p},L^{\infty})}(I^{n})}&\approx\bigg\|\int^{1}_{t^{n^{\prime}}}f^{*}(s)s^{-1/n^{\prime}}ds\bigg\|_{L^{p(n-1)/(n-p),p}(0,1)}\\
&\approx \bigg(\int^{1}_{0}t^{(n-p)/n-1}\bigg(\int^{1}_{t}f^{*}(s)s^{-1/n^{\prime}}ds\bigg)^{p}dt\bigg)^{1/p}.
\end{align*}
Using now Hardy's inequalities  we obtain
\begin{align}\label{eq: 1WLp->R(X,Loo) 1<p<n}
\|f\|_{Z_{\mathcal{R}(L^{p(n-1)/(n-p),p},L^{\infty})}(I^{n})}\lesssim \|f\|_{L^{p}(I^{n})}.
\end{align}
On the other hand, again Hardy's inequalities give us that 
\begin{align*}
\|f\|_{Z_{\mathcal{R}(L^{p(n-1)/(n-p),p},L^{\infty})}(I^{n})}&\gtrsim \bigg(\int^{1}_{0}t^{-(n(p-1)+p)/n-1}\bigg(\int^{t}_{0}\int^{1}_{v}f^{*}(s)s^{-1/n^{\prime}}ds\,dv\bigg)^{p}dt\bigg)^{1/p}\\
&=\bigg(\int^{1}_{0}t^{-p(n-1)/n}\bigg(\int^{t}_{0}\int^{1}_{v}f^{*}(s)s^{-1/n^{\prime}}ds\,dv\bigg)^{p}dt\bigg)^{1/p}.
\end{align*}
But, if $0<t<1$, then
\begin{align*}
\int^{t}_{0}\int^{1}_{v}f^{*}(s)s^{-1/n^{\prime}}ds\,dv&=t\int^{1}_{t}f^{*}(s)s^{-1/n^{\prime}}ds+\int^{t}_{0}f^{*}(v)v^{1/n}dv\gtrsim f^{*}(t)t^{1/n+1}.
\end{align*}
So,
\begin{align}\label{eq: 2WLp->R(X,Loo) 1<p<n}
\|f\|_{Z_{\mathcal{R}(L^{p(n-1)/(n-mp),p},L^{\infty})}(I^{n})}\gtrsim \|f\|_{L^{p}(I^{n})}.
\end{align}
Thus, combining \eqref{eq: 1WLp->R(X,Loo) 1<p<n} and \eqref{eq: 2WLp->R(X,Loo) 1<p<n}, we obtain 
\begin{align*}
L^{p}(I^{n})=Z_{\mathcal{R}(L^{p(n-1)/(n-p),p},L^{\infty})}(I^{n}),
\end{align*} 
as we wanted to show.
\end{proof}

Finally, we shall see that a nontrivial improvement of the domain in \eqref{eq: W1Ln->R(X,Loo) rango} is possible among r.i.\ spaces. For this, we will use the following result~\cite{Bennett-Rudnick}:

\begin{lemma}\label{lema: indice Boyd Lorentz-Zygmund} If $1<p< \infty,$ then $\overline{\alpha}_{L^{\infty,p;-1}}=0.$
\end{lemma}

\begin{corollary}
 The  r.i.\ space $Z_{\mathcal{R}(L^{\infty,n;-1},L^{\infty})}(I^{n}),$ with norm given by
\begin{align*}
\|f\|_{Z_{\mathcal{R}(L^{\infty,n;-1},L^{\infty})}(I^{n})}\approx\bigg\|\int^{1}_{t}s^{-1/n^{\prime}}f^{*}(s)ds\bigg\|_{L^{\infty,n;-1}(I^{n})},
\end{align*}
is the largest r.i.\ domain space that verifies
\begin{align*}
W^{1}Z_{\mathcal{R}(L^{\infty,n;-1},L^{\infty})}(I^{n})\hookrightarrow \mathcal{R}(L^{\infty,n;-1},L^{\infty}).
\end{align*}
\end{corollary}

\begin{proof}According to Lemma~\ref{lema: indice Boyd Lorentz-Zygmund}, we may apply Theorem~\ref{teo: W1Z->R(X,Loo) dominio} to obtain
\begin{align*}
\|f\|_{Z_{\mathcal{R}(L^{\infty,n;-1},L^{\infty})}(I^{n})}&\approx\bigg\|\int^{1}_{t^{n^{\prime}}}s^{-1/n^{\prime}}f^{*}(s)ds\bigg\|_{L^{\infty,n;-1}(I^{n})}.
\end{align*} 
Then, the result follows using a change of variables  and  \cite[Theorem 3.1]{Carro-Pick-Soria-Stepanov}.
\end{proof}

\section{Comparison with the optimal r.i. space}\label{sec: Comparison with the optimal r.i. space}

As we have mentioned before, Kerman and Pick  \cite{Kerman-Pick} studied the optimal range problem for Sobolev embedding within the class of r.i.\ spaces. Namely, for a fixed r.i.\ domain space $Z(I^{n})$, they determined  the smallest r.i.\ space  $X^{\textnormal{op}}(I^{n}),$ satisfying 
	\begin{align}\label{eq: Lubos optimo}
W^{1}Z(I^{n})\hookrightarrow X^{\textnormal{op}}(I^{n}).
\end{align}

In our setting, we recall that in Theorem \ref{teo: W1Z->R(X,Loo) rango} we have studied an analogous problem in the context of   mixed norm spaces. More precisely,  we have found the smallest space of the form $\mathcal{R}(X, L^{\infty}),$ namely $\mathcal{R}(X_{W^{1}Z,L^{\infty}},L^{\infty}),$ that verifies 
 \begin{align}\label{eq: mixtas optimo}
 W^{1}Z(I^{n})\hookrightarrow \mathcal{R}(X_{W^{1}Z,L^{\infty}},L^{\infty}).
 \end{align}
 
Now, our goal is to compare the optimal r.i.\ range space  with the optimal mixed norm space. We will show in Theorem~\ref{theorem: R(X,Loo)->Lubos} that  the  following chain of embeddings holds:
 \begin{align*}
 W^{1}Z(I^{n})\hookrightarrow \mathcal{R}(X_{W^{1}Z,L^{\infty}},L^{\infty})\hookrightarrow  X^{\textnormal{op}}(I^{n}).
 \end{align*}

To this end, we first need to recall the following result~\cite{Kerman-Pick}.

\begin{theorem}\label{theorem: reduccion Kerman-Pick} Let $Y(I^{n})$ and $Z(I^{n})$ be  r.i.\ spaces. Then, the Sobolev embedding
\begin{align*}
W^{1}Z(I^{n})\hookrightarrow Y(I^{n})
\end{align*}
holds if and only if 
\begin{align*}
 \bigg\|\int^{1}_{t}t^{1/n-1}f(t)dt\bigg\|_{\overline{Y}(0,1)}\lesssim \|f\|_{\overline{Z}(0,1)},\ \ f\in \overline{Z}(0,1).
 \end{align*}
\end{theorem} 

\begin{remark}\label{remark: Kerman-Pick}\textnormal{We would also like to emphasize that Theorem~\ref{theorem: reduccion Kerman-Pick} was  used  in  \cite{Kerman-Pick}  to prove Sobolev estimates as well as to give the following characterization  of the optimal range space when the domain space is given:
\begin{align}\label{eq: optimal range Kerman-Pick}
(X^{\textnormal{op}})^{\prime}(I^{n})=\Big\{f\in\mathcal{M}(I^{n}): \|f\|_{(X^{\textnormal{op}})^{\prime}(I^{n})}=\big\|t^{1/n}f^{**}(t)\big\|_{\overline{Z}^{\prime}(0,1)}<\infty\Big\}.
\end{align}
As a consequence, for instance, they recovered   the classical estimates by  Poornima~\cite{Poornima}, O'Neil~\cite{Oneil-1963} and Peetre~\cite{Peetre}
 \begin{align}\label{eq: intro WLp->Xop}
W^{1}L^{p}(I^{n})\hookrightarrow L^{np/(n-1p),p}(I^{n}),
\end{align}
and  the so-called limiting or critical case of Sobolev embedding due to  Hansson~\cite{Hansson},  Brezis and Wainger~\cite{Brezis-Wainger} and Maz'ya~\cite{Mazya}
\begin{align}\label{eq: intro WLn->Xop}
W^{1}L^{n}(I^{n})\hookrightarrow L^{\infty,n;-1}(I^{n}).
\end{align}
Furthermore, as a new contribution, the authors showed that the range spaces $L^{np/(n-p),p}(I^{n})$ and $L^{\infty,n;-1}(I^{n})$ in \eqref{eq: intro WLp->Xop} and \eqref{eq: intro WLn->Xop} respectively, are the best possible among r.i.\ spaces. We now see that we can further improve these results.}
\end{remark}

\begin{theorem}\label{theorem: R(X,Loo)->Lubos} Let $Z(I^{n})$ be an r.i.\ space, let $X^{\textnormal{op}}(I^{n})$ be the optimal r.i.\   space in \eqref{eq: Lubos optimo} and let $\mathcal{R}(X_{W^{1}Z,L^{\infty}},L^{\infty})$ be the smallest  space of the form $\mathcal{R}(X,L^{\infty})$ that verifies \eqref{eq: mixtas optimo}. Then,
		\begin{align*}
W^{1}Z(I^{n})\hookrightarrow \mathcal{R}(X_{W^{1}Z,L^{\infty}},L^{\infty}) \hookrightarrow X^{\textnormal{op}}(I^{n}).
	\end{align*}
	Moreover, $X^{\textnormal{op}}(I^{n})$ is the smallest r.i.\ space that verifies
	\begin{align*}
	 \mathcal{R}(X_{W^{1}Z,L^{\infty}},L^{\infty})\hookrightarrow  X^{\textnormal{op}}(I^{n}).
	\end{align*}
\end{theorem}

\begin{proof} We fix $\mathcal{R}(X_{W^{1}Z,L^{\infty}},L^{\infty}).$ Using now Theorem \ref{teo: rango R(X,Loo)->Z}, we construct the smallest r.i.\ space, denoted by $Y_{\mathcal{R}(X_{W^{1}Z,L^{\infty}},L^{\infty})}(I^{n}),$  that verifies 
  	\begin{align}\label{eq: 2R(X,Loo)->Lubos}
\mathcal{R}(X_{W^{1}Z,L^{\infty}},L^{\infty}) \hookrightarrow Y_{\mathcal{R}(X_{W^{1}Z,L^{\infty}},L^{\infty})}(I^{n}).
\end{align}
Then, by \eqref{eq: mixtas optimo}, it follows that
\begin{align*}
W^{1}Z(I^{n})\hookrightarrow Y_{\mathcal{R}(X_{W^{1}Z,L^{\infty}},L^{\infty})}(I^{n}),
\end{align*}
and hence, our assumption on $X^{\textnormal{op}}(I^{n})$ implies that
\begin{align}\label{eq: 3R(X,Loo)->Lubos}
  X^{\textnormal{op}}(I^{n})\hookrightarrow Y_{\mathcal{R}(X_{W^{1}Z,L^{\infty}},L^{\infty})}(I^{n}).
\end{align}

On the other hand, if $f\in  X^{\textnormal{op}}(I^{n}),$ then, using a change of variables, we get
\begin{align*}
\|f\|_{X^{\textnormal{op}}(I^{n})}&=\sup_{\|g\|_{(X^{\textnormal{op}})^{\prime}(I^{n})}\leq 1}\int^{1}_{0}f^{*}(t)g^{*}(t)dt\\
&\leq \sup_{\|g\|_{(X^{\textnormal{op}})^{\prime}(I^{n})}\leq 1}\int^{1}_{0}f^{*}(t)\,t^{-1/n}\sup_{t<s<1}s^{1/n}g^{*}(s)dt\\
&\approx \sup_{\|g\|_{(X^{\textnormal{op}})^{\prime}(I^{n})}\leq 1} \int^{1}_{0}f^{*}(t^{n^{\prime}})\sup_{t^{n^{\prime}}<s<1}s^{1/n}g^{*}(s)dt
\end{align*}
and hence,  by  H\"older's inequality, we obtain that 
\begin{align}\label{eq: 4R(X,Loo)->Lubos}
\|f\|_{X^{\textnormal{op}}(I^{n})}\lesssim  \sup_{\|g\|_{(X^{\textnormal{op}})^{\prime}(I^{n})}\leq 1}&\|f^{*}(t^{n^{\prime}})\|_{\overline{X}_{W^{1}Z,L^{\infty}}(0,1)}\bigg\|\sup_{t^{n^{\prime}}<s<1}s^{1/n}g^{*}(s)\bigg\|_{\overline{X}^{\prime}_{W^{1}Z,L^{\infty}}(0,1)}.\end{align}
But, combining Theorem~\ref{teo: W1Z->R(X,Loo) rango} and \eqref{eq: optimal range Kerman-Pick}, we have that
\begin{align*}
\bigg\|\sup_{t^{n^{\prime}}<s<1}s^{1/n}g^{*}(s)\bigg\|_{\overline{X}^{\prime}_{W^{1}Z,L^{\infty}}(0,1)}&\approx\bigg\|t^{-1/n^{\prime}}\int^{t}_{0}y^{-1/n}\sup_{y<s<1}s^{1/n}g^{*}(s)dy\bigg\|_{\overline{Z}^{\prime}(0,1)}\\
&=\|t^{-1/n}\sup_{t<s<1}s^{1/n}g^{*}(s)\|_{\overline{X^{\textnormal{op}}}^{\prime}(0,1)},
\end{align*}
and so, by \cite[Remark~3.11]{Kerman-Pick}, we deduce that
\begin{align*}
\bigg\|\sup_{t^{n^{\prime}}<s<1}s^{1/n}g^{*}(s)\bigg\|_{\overline{X}^{\prime}_{W^{1}Z,L^{\infty}}(0,1)}\lesssim \|g\|_{(X^{\textnormal{op}})^{\prime}(I^{n})}.
\end{align*}
Therefore, using this fact  and \eqref{eq: 4R(X,Loo)->Lubos}, we get
\begin{align*}
\|f\|_{X^{\textnormal{op}}(I^{n})}\lesssim \|f^{*}(t^{n^{\prime}})\|_{\overline{X}_{W^{1}Z,L^{\infty}}(0,1)}.
\end{align*}
 and hence, by Theorem \ref{teo: rango R(X,Loo)->Z}, we conclude that 
\begin{align*}
\|f\|_{X^{\textnormal{op}}(I^{n})}\lesssim \|f\|_{Y_{\mathcal{R}(X_{W^{1}Z,L^{\infty}},L^{\infty})}(I^{n})}, 
\end{align*}
 from which it follows that
 \begin{align}\label{eq: 5R(X,Loo)->Lubos}
  Y_{\mathcal{R}(X_{W^{1}Z,L^{\infty}},L^{\infty})}(I^{n})\hookrightarrow X^{\textnormal{op}}(I^{n}).
 \end{align}
As a consequence, combining \eqref{eq: 3R(X,Loo)->Lubos} and \eqref{eq: 5R(X,Loo)->Lubos} yields
 \begin{align*}
   X^{\textnormal{op}}(I^{n})=Y_{\mathcal{R}(X_{W^{1}Z,L^{\infty}},L^{\infty})}(I^{n}),
 \end{align*}
 and so,  \eqref{eq: mixtas optimo} and  \eqref{eq: 2R(X,Loo)->Lubos} imply that
 		\begin{align*}
W^{1}Z(I^{n})\hookrightarrow \mathcal{R}(X_{W^{1}Z,L^{\infty}},L^{\infty}) \hookrightarrow X^{\textnormal{op}}(I^{n})=Y_{\mathcal{R}(X_{W^{1}Z,L^{\infty}},L^{\infty})}(I^{n}),
	\end{align*}
as we wanted to show.
 \end{proof}

As a consequence of Theorem~\ref{theorem: R(X,Loo)->Lubos}, we shall see  that  the classical estimates for the standard Sobolev space $W^{1}L^{p}(I^{n})$ by Poornima~\cite{Poornima}, O'Neil~\cite{Oneil-1963} and Peetre~\cite{Peetre} $(1\leq p < n)$, and by Hansson~\cite{Hansson} and Brezis and Wainger~\cite{Brezis-Wainger} and  Maz'ya~\cite{Mazya} ($p=n$) can be improved  considering mixed norms on the target spaces.

\begin{corollary}\label{corollary: W1Lp->R(X,Loo)->Lubos} Let $1\leq p<n.$ Then, 
\begin{align*}
W^{1}L^{p}(I^{n})\hookrightarrow \mathcal{R}(L^{p(n-1)/(n-p),p}, L^{\infty})\underset{\neq}{\hookrightarrow}L^{pn/(n-p),p}(I^{n}).
\end{align*}
\end{corollary}

\begin{proof}If $Z(I^{n})=L^{p}(I^{n}),$ with $1\leq p<n,$ then,  by Remark~\ref{remark: Kerman-Pick} and Corollary~\ref{corollary: W1Lp->R(X,Loo) rango}, we have
\begin{align*}
X^{\textnormal{op}}(I^{n})=L^{pn/(n-p),p}(I^{n})\ \   \textnormal{and}\  \   \mathcal{R}(X_{W^{1}Z,L^{\infty}},L^{\infty})=\mathcal{R}(L^{p(n-1)/(n-p),p}, L^{\infty}).
\end{align*}
Therefore, using Theorem~\ref{theorem: R(X,Loo)->Lubos} and Theorem~\ref{teo: mixta inclusion Z->R(X,Loo)}  , the result follows immediately.
\end{proof}

\begin{corollary} $W^{1}L^{n}(I^{n})\hookrightarrow \mathcal{R}(L^{\infty,n ;-1},L^{\infty})\underset{\neq}{\hookrightarrow}L^{\infty,n ;-1}(I^{n}).$
\end{corollary}

\begin{proof}Use Corollary~\ref{corollary: W1Ln->R(X,Loo) rango} instead of Corollary~\ref{corollary: W1Lp->R(X,Loo) rango} and argue as in the proof of Corollary~\ref{corollary: W1Lp->R(X,Loo)->Lubos}.
\end{proof}

\noindent
{\bf Acknowledgments:} We would like to thank the referee for his/her careful revision which has improved the final version of this work.

\providecommand{\bysame}{\leavevmode\hbox to3em{\hrulefill}\thinspace}
\providecommand{\MR}{\relax\ifhmode\unskip\space\fi MR }
\providecommand{\MRhref}[2]{%
  \href{http://www.ams.org/mathscinet-getitem?mr=#1}{#2}
}
\providecommand{\href}[2]{#2}


\begin{thebibliography}{10}

\bibitem{Adams}
R.~A. Adams and J.~J.~F. Fournier, \emph{Sobolev Spaces}, second ed., Pure and
  Applied Mathematics (Amsterdam), vol. 140, Elsevier/Academic Press,
  Amsterdam, 2003.

\bibitem{Robert-Viktor}
R.~Algervik and V.~I. Kolyada, \emph{On {F}ournier-{G}agliardo mixed norm
  spaces}, Ann. Acad. Sci. Fenn. Math. \textbf{36} (2011), no.~2, 493--508.

\bibitem{Barza-Kaminska-Persson-Soria}
S.~Barza, A.~Kami{\'n}ska, L.~Persson, and J.~Soria, \emph{Mixed norm and
  multidimensional {L}orentz spaces}, Positivity \textbf{10} (2006), no.~3,
  539--554.

\bibitem{Benedek-Panzone}
A.~Benedek and R.~Panzone, \emph{The space {$L^{p}$}, with mixed norm}, Duke
  Math. J. \textbf{28} (1961), 301--324.
  
\bibitem{Bennett-Rudnick} 
C.~Bennett  and K. Rudnick, \emph{On Lorentz-Zygmund spaces}, Dissertationes Math. \textbf{175} (1980), 1--72.

\bibitem{Bennett}
C.~Bennett and R.~Sharpley, \emph{Interpolation of {O}perators}, Pure and
  Applied Mathematics, vol. 129, Academic Press Inc., Boston, MA, 1988.
  
\bibitem{Bergh}
J.~Bergh and J.~L{\"o}fstr{\"o}m, \emph{Interpolation Spaces. {A}n Introduction},
  Springer-Verlag, Berlin, 1976.

\bibitem{Blei-Fournier}
R.~C. Blei and J.~J.~F. Fournier, \emph{Mixed-norm conditions and {L}orentz
  norms}, Commutative harmonic analysis ({C}anton, {NY}, 1987), vol.~91, 1989,
  pp.~57--78.

\bibitem{Blozinski}
A.~P. Blozinski, \emph{Multivariate rearrangements and {B}anach function spaces
  with mixed norms}, Trans. Amer. Math. Soc. \textbf{263} (1981), 149--167.

\bibitem{Boccuto-Bukhvalov-Sambucini}
A.~Boccuto, A.~V. Bukhvalov, and A.~R. Sambucini, \emph{Some inequalities in
  classical spaces with mixed norms}, Positivity \textbf{6} (2002), 393--411.

\bibitem{Brezis}
H.~Brezis, \emph{Analyse Fonctionnelle}, Collection Math\'ematiques Appliqu\'es pour la Ma\^itrise, Masson, Paris, 1983.
  
 \bibitem{Brezis-Wainger}
H.~Brezis and S.~Wainger, \emph{A note on limiting cases of {S}obolev
  embeddings and convolution inequalities}, Comm. Partial Differential
  Equations \textbf{5} (1980), no.~7, 773--789.

\bibitem{Buhvalov}
A.~V. Buhvalov, \emph{Spaces with mixed norm}, Vestnik Leningrad. Univ. (1973),
  no.~19 Mat. Meh. Astronom. Vyp. 4, 5--12, 151.

\bibitem{Carro-Pick-Soria-Stepanov}
M.~J. Carro, L.~Pick, J.~Soria, and V.~D. Stepanov, \emph{On embeddings between
  classical {L}orentz spaces}, Math. Inequal. Appl. \textbf{4} (2001),
  397--428.
  
\bibitem{Cianchi}
A. Cianchi, \emph{Symmetrization and second-order {S}obolev inequalities}, Ann.
  Mat. Pura Appl. (4) \textbf{183} (2004), no.~1, 45--77.

  
 \bibitem{Clavero-Soria}
N.~Clavero and J.~Soria, \emph{Mixed norm spaces and rearrangement invariant estimates}, J. Math. Anal. Appl. \textbf{419} (2014), 878--903.

\bibitem{DeVore-Scherer}
R.~DeVore and K.~Scherer, \emph{Interpolation of linear operators on {S}obolev
  spaces}, Ann. of Math. \textbf{109} (1979), 583--599.

\bibitem{Edmunds-Kerman-Pick}
D.~E. Edmunds, R.~Kerman, and L.~Pick, \emph{Optimal {S}obolev imbeddings
  involving rearrangement-invariant quasinorms}, J. Funct. Anal. \textbf{170}
  (2000), 307--355.

\bibitem{Fournier}
J.~J.~F. Fournier, \emph{Mixed norms and rearrangements: {S}obolev's inequality
  and {L}ittlewood's inequality}, Ann. Mat. Pura Appl. \textbf{148} (1987),
  51--76.

\bibitem{Gagliardo}
E.~Gagliardo, \emph{Propriet\`a di alcune classi di funzioni in pi\`u
  variabili}, Ricerche Mat. \textbf{7} (1958), 102--137.
  
\bibitem{Hansson}
K.~Hansson, \emph{Imbedding theorems of {S}obolev type in potential theory},
  Math. Scand. \textbf{45} (1979), no.~1, 77--102.  

\bibitem{Holmstedt}
T.~Holmstedt, \emph{Interpolation of quasi-normed spaces}, Math. Scand.
  \textbf{26} (1970), 177--199.

\bibitem{Kerman-Pick}
R.~Kerman and L.~Pick, \emph{Optimal {S}obolev imbeddings}, Forum Math.
  \textbf{18} (2006), no.~4, 535--570.

\bibitem{Viktor2012}
V.~I. Kolyada, \emph{Iterated rearrangements and {G}agliardo-{S}obolev type
  inequalities}, J. Math. Anal. Appl. \textbf{387} (2012), 335--348.

\bibitem{Viktor2013}
\bysame, \emph{On {F}ubini type property in {L}orentz spaces}, Recent
  Advances in Harmonic Analysis and Applications \textbf{25} (2013), 171--179.

\bibitem{Mazya}
V.~Maz'ya, \emph{Sobolev Spaces}, Springer Series in Soviet Mathematics, Springer-Verlag, Berlin, 1985.

\bibitem{Milman-kfunctional-mixed-norm}
M. Milman, \emph{Notes on interpolation of mixed norm spaces and applications},
  Quart. J. Math. Oxford Ser. (2) \textbf{42} (1991), no.~167, 325--334.
  
\bibitem{Nirenberg}
L.~Nirenberg, \emph{On elliptic partial differential equations}, Ann. Scuola
  Norm. Sup. Pisa \textbf{13} (1959), 115--162.

\bibitem{Oneil-1963}
R.~O'Neil, \emph{Convolution operators and {$L(p,\,q)$} spaces}, Duke Math. J.
  \textbf{30} (1963), 129--142.

  \bibitem{Peetre}
J.~Peetre, \emph{Espaces d'interpolation et th\'eor\`eme de {S}oboleff}, Ann.
  Inst. Fourier (Grenoble) \textbf{16} (1966), no.~fasc. 1, 279--317.
  
    \bibitem{Poornima}
S.~Poornima, \emph{An embedding theorem for the Sobolev space $W^{1,1}$}, Bull. Sci. Math. (2) \textbf{107} (1983), no. 3, 253--259.

\bibitem{Sobolev}
S.~L. Sobolev, \emph{On a theorem of functional analysis}, Math. Sb.
  \textbf{46} (1938), 471--496, translated in Amer. Math. Soc. Transl.  \textbf{34}
  (1963), 39--68.

\bibitem{Talenti}
G.~Talenti, \emph{Inequalities in rearrangement invariant function spaces},
  Nonlinear analysis, function spaces and applications, {V}ol.\ 5 ({P}rague,
  1994), Prometheus, Prague, 1994, pp.~177--230.
  

\end{thebibliography}
\end{document}